\documentclass[a4paper]{amsart}
\usepackage{amssymb}
\usepackage{amscd}
\usepackage{amsthm}
\usepackage{amsmath}
\usepackage{latexsym}
\usepackage[all]{xy}
\usepackage[czech,english]{babel}
\theoremstyle{plain}
\newtheorem{theorem}{Theorem}
\newtheorem{corollary}[theorem]{Corollary}
\newtheorem{lemma}[theorem]{Lemma}
\newtheorem{proposition}[theorem]{Proposition}

\theoremstyle{definition}

\newtheorem{example}[theorem]{Example}

\theoremstyle{remark}

\newtheorem{remark}[theorem]{Remark}

\numberwithin{theorem}{section}

\newcommand{\mSpec}[1]{\mbox{\rm{mSpec}}(#1)}

\newcommand{\Add}{\mbox{\rm{Add\,}}}
\newcommand{\Prod}{\mbox{\rm{Prod\,}}}

\newcommand{\Cog}{\mbox{\rm{Cogen\,}}}

\newcommand{\Ass}[1]{\mbox{\rm{Ass}} \,#1}

\newcommand{\Hom}[3]{\mbox{\rm{Hom}}_{#1}(#2,#3)}
\newcommand{\Ext}[4]{\mbox{\rm{Ext}}^{#1}_{#2}(#3,#4)}
\newcommand{\Tor}[4]{\mbox{\rm{Tor}}_{#1}^{#2}(#3,#4)}

\newcommand{\rmod}[1]{\mbox{\rm{Mod}--}{#1}}

\newcommand{\ModR}{\text{Mod-}R}

\newcommand{\pd}[1]{\mbox{\rm{proj.dim\,}}#1}

\newcommand{\Ker}[1]{\mbox{\rm{Ker}}(#1)}
\newcommand{\Coker}[1]{\mbox{\rm{Coker}}(#1)}

\begin{document}

\title{Test sets for factorization properties of modules}
\author{\textsc{Jan \v{S}aroch and Jan Trlifaj}}
\address{Charles University, Faculty of Mathematics and Physics, Department of Algebra \\
Sokolovsk\'{a} 83, 186 75 Prague 8, Czech Republic}
\email{saroch@karlin.mff.cuni.cz}
\email{trlifaj@karlin.mff.cuni.cz}

\date{\today}
\dedicatory{To L\'{a}szl\'{o} Fuchs in honor of his 95th birthday}
\thanks{Research supported by GA\v CR 17-23112S}
\begin{abstract} Baer's Criterion of injectivity implies that injectivity of a module is a factorization property w.r.t.\ a single monomorphism. Using the notion of a cotorsion pair, we study generalizations and dualizations of factorization properties in dependence on the algebraic structure of the underlying ring $R$ and on additional set-theoretic hypotheses. For $R$ commutative noetherian of Krull dimension $0 < d < \infty$, we show that the assertion {\lq}projectivity is a~factorization property w.r.t.\ a single epimorphism{\rq} is independent of ZFC + GCH. We also show that if $R$ is any ring and there exists a strongly compact cardinal $\kappa > |R|$, then the category of all projective modules is accessible.              
\end{abstract}

\maketitle

Various classes of modules studied in homological algebra are defined by factorization properties with respect to proper classes of monomorphisms or epimorphisms. A natural question arises of whether these defining classes of morphisms can be replaced by sets, and hence by single morphisms. In many cases, this is possible. We start with the easier setting of injective modules, and more in general, right-hand classes of complete cotorsion pairs. Then we will turn to the more difficult setting of projective modules, and left-hand classes of cotorsion pairs, where solutions often require extra algebraic and/or set-theoretic assumptions. Our notation will follow \cite{EM} and \cite{GT}.    

\section{Injective modules and their generalizations}

Injective modules and their various generalizations are defined by factorization properties with respect to proper classes of monomorphisms. Denote by $\Psi$ the class of all monomorphisms in $\rmod R$. Then a module $M$ is \emph{injective}, iff $\Hom R{\psi}M$ is surjective for each $\psi \in \Psi$, i.e., iff for each $\psi \in \Psi$ with $\psi \in \Hom RN{N^\prime}$, each homomorphism from $N$ to $M$ factorizes through $\psi$. The classic Baer Criterion for injectivity \cite{B} says that in the definition above, one can replace the proper class $\Psi$ by its subset, and hence by a single element. Namely, if $S$ denotes the set of all right ideals of $R$, then $M$ is injective, iff $\Hom R{\varphi }M$ is surjective, where $\varphi : \bigoplus_{I \in S} I \hookrightarrow \bigoplus_{I \in S} R$. By dimension shifting, it follows that for each $n < \omega$, there is a single monomorhism $\varphi_n$ such that factorization through $\varphi_n$ defines the property of having injective dimension $\leq n$. 

Given a monomorphism $\psi$ in $\rmod R$, we call a class of modules $\mathcal C$ a \emph{factorization class} of $\psi$ provided that $\mathcal C = \{ M \in \rmod R \mid \Hom R{\psi}M \hbox{ is surjective} \}$. So, for example, the class $\mathcal I _n$ of all modules of injective dimension $\leq n$ is the factorization class of $\psi_n$. Using Enochs' proof of the Flat Cover Conjecture (i.e., the fact that the class $\mathcal F _0$ of all flat modules is deconstructible), we infer that there is a monomorphism $\nu$ such that the class $\mathcal {EC}$ of all Enochs cotorsion modules is the factorization class of~$\nu$. 

In order to see that these results are particular instances of a more general phenomenon, we need to recall several definitions:

Let $\mathcal S$ be a class of modules. A module $M$ is \emph{$\mathcal S$-filtered} provided that $M$ is the union of an increasing well-ordered continuous chain of its submodules $( M _\alpha, \nu_{\beta \alpha} \mid \alpha < \beta < \sigma)$ such that $\nu_{\beta \alpha} : M_\alpha \to M_\beta$ is the inclusion for all $\alpha < \beta < \sigma$, $M_0 = 0$, and for each $\alpha < \sigma$, $\Coker{\nu_{\alpha+1 \alpha}}$ is isomorphic to an element of $\mathcal S$. The class of all $\mathcal S$-filtered modules is denoted by $\hbox{Filt}(\mathcal S)$. The class of all direct summands of $\mathcal S$-filtered modules is denoted by $\hbox{sFilt}(\mathcal S)$. Note that $\hbox{Filt}(\hbox{Filt}(\mathcal S)) = \hbox{Filt}(\mathcal S)$ and $\hbox{Filt}(\hbox{sFilt}(\mathcal S)) = \hbox{sFilt}(\mathcal S)$. In particular, both $\hbox{Filt}(\mathcal S)$ and $\hbox{sFilt}(\mathcal S)$ are closed under arbitrary direct sums and extensions.

For example, if $\mathcal S = \{ R \}$, then $\hbox{Filt}(\mathcal S)$ is the class of all free modules, and $\hbox{sFilt}(\mathcal S)$ the class of all projective modules.

A class of modules $\mathcal C$ is \emph{deconstructible} provided that there is a \emph{subset} $\mathcal S \subseteq \mathcal C$ such that $\mathcal C = \hbox{Filt}(\mathcal S)$.     

For a class of modules $\mathcal C$, we define $\mathcal C ^\perp = \{ M \in \rmod R \mid \Ext 1RCM = 0 \hbox{ for all } C \in \mathcal C\}$ and $^\perp \mathcal C  = \{ M \in \rmod R \mid \Ext 1RMC = 0 \hbox{ for all } C \in \mathcal C\}$. A~pair of classes of modules $\mathfrak C = (\mathcal A, \mathcal B)$ is a \emph{cotorsion pair} in case $\mathcal A = {}^\perp \mathcal B$ and $\mathcal B = \mathcal A ^\perp$. $\mathfrak C$ is said to be generated (cogenerated) by a class $\mathcal C$, if $\mathcal B = \mathcal C ^\perp$ ($\mathcal A = {}^\perp \mathcal C$). $\mathfrak C$ is \emph{hereditary}, if $\Ext iRAB = 0$ for all $A \in \mathcal A$, $B \in \mathcal B$ and all $i \geq 2$.      

\begin{lemma}\label{setgen} Let $\mathfrak C = (\mathcal A,\mathcal B)$ be a cotorsion pair in $\rmod R$. Then the following conditions are equivalent:
\begin{enumerate}
\item $\mathfrak C$ is generated by a set.
\item There is a subset $\mathcal S \subseteq \mathcal A$ such that $\mathcal A = \hbox{sFilt}(\mathcal S)$.
\item $\mathcal A$ is deconstructible.
\item There is a monomorphism $f : Q \hookrightarrow P$ with $P$ projective such that $\mathcal B$ is the factorization class of $f$.
\end{enumerate}
\end{lemma}
\begin{proof} (1) implies (3) by the Kaplansky Theorem for cotorsion pairs \cite[7.13]{GT}. Since $\mathcal A$ is closed under direct summands, (3) implies (2). That (2) implies (1) follows by the Eklof Lemma \cite[6.2]{GT}.  

Assume (1). Then there is a set $\mathcal S \subseteq \mathcal A$ such that $\mathcal B = \mathcal S ^\perp$. For each $A \in \mathcal S$, there is a presentation $0 \to Q_A \overset{\nu_A}\to P_A \to A \to 0$ with $P_A$ projective. We define a monomorphism $f$ by $f = \bigoplus_{A \in \mathcal S} \nu_A$. Then $f \in \Hom RQP$ where $Q = \bigoplus_{A \in \mathcal S} Q_A$ and $P = \bigoplus_{A \in \mathcal S} P_A$ is projective. Moreover, for each $N \in \rmod R$, each homomorphism from $Q$ to $N$ factorizes through $f$, iff $\Ext 1RAN = 0$ for each $A \in \mathcal S$, iff $N \in \mathcal B$. So (4) holds.

Assume (4). Then for each $N \in \rmod R$, $N \in \mathcal B$, iff $\Ext 1R{P/Q}N = 0$. That is, $\mathfrak C$ is generated by the one element set $\{ P/Q \}$, proving (1).
\end{proof}          

For integral domains, other variants of cotorsion modules are studied. Lemma \ref{setgen} applies to them, too:  

\begin{corollary}\label{MatWar} Let $R$ be an integral domain. Then there exist monomorphisms $g_0 \in \Hom R{Q_0}{P_0}$ and $g_1 \in \Hom R{Q_1}{P_1}$ with $P_0$ and $P_1$ projective, such that the class of all Matlis cotorsion (Warfield cotorsion) modules is the factorization class of $g_0$ ($g_1$).
\end{corollary}
\begin{proof} The first claim follows by Lemma \ref{setgen}, since a module $M$ is Matlis cotorsion, iff $\Ext 1RQM = 0$ where $Q$ denotes the quotient field of $R$. For the second claim, one uses the characterization of Warfield cotorsion modules as the Matlis cotorsion modules of injective dimension $\leq 1$, \cite[7.47]{GT}.  
\end{proof}

A different, and more involved, argument shows that there is a single pure monomorphism $\mu$ such that the class of all pure-injective modules is the factorization class of $\mu$:

\begin{lemma}\label{pure} Let $\kappa = |R| + \aleph_0$. Then a module $L$ is pure-injective, iff for each pure embedding $\nu \in \Hom RNM$ with $|M| \leq \kappa$ and each $f \in \Hom RNL$ there exists $g \in \Hom RML$ with $g \nu = f$.
\end{lemma} 
\begin{proof} The only-if part is trivial, let us prove the if part. By \cite[Theorem~V.1.2]{EM}, it is enough to show that each system $\mathcal S$ consisting of $R$-linear equations (in any number, finite or infinite, of variables) with parameters in $L$ has a solution in $L$ provided that $|\mathcal S|\leq\kappa$ and each finite $\mathcal F\subseteq\mathcal S$ has a solution in $L$.

Let $\mathcal S$ be such a system. Fix a submodule $N\leq L$ of cardinality $\leq\kappa$ such that $N$ contains all the parameters from $\mathcal S$ and each finite $\mathcal F\subseteq \mathcal S$ has a solution in $N$. Let us denote by $f$ the inclusion $N\subseteq L$. Following the proof of \cite[Theorem~V.1.2\ $(2)\Rightarrow (3)$]{EM}, we put $M = (N\oplus F)/K$ where $F$ is the free modules whose basis is the set of variables in the system $\mathcal S$, and $K$ is the submodule generated by all the elements $(-b,\sum_{j=1}^m y_j r_j)$ where $\sum_{j=1}^m y_j r_j = b$ belongs to $\mathcal S$. Let $\iota: N\hookrightarrow M$ be the canonical embedding. Using finite solvability of $\mathcal S$ in $N$, it is straightforward to check that $\iota$ is actually a pure embedding. Since $|M|\leq \kappa$, we can use our assumption to find $g \in \Hom RML$ such that $g \iota = f$. By construction, $\mathcal S$ has a~solution in $M$. The $g$-image of this solution is a desired solution of $\mathcal S$ in $L$.
\end{proof}

If we consider $\lambda$-purity and $\lambda$-pure-injective modules for a regular uncountable cardinal $\lambda$, the situation is rather unclear. The difficulty here stems from the fact that $\ModR$ often does not possess enough $\lambda$-pure-injective modules. Recall that a~$\lambda$-accessible category $\mathcal A$ \emph{has enough $\lambda$-pure-injective objects} if each object $A$ in $\mathcal A$ can be $\lambda$-purely embedded into a $\lambda$-pure-injective object. For the definition of $\lambda$-purity in this generality, we refer to \cite[2.27]{AR}.

For our purposes, a monomorphism $f:B\to A$ is \emph{$\lambda$-pure} if the homomorphism $\Hom RCA\to \Hom RC{\Coker f}$ induced by the canonical projection $A\to \Coker f$ is surjective for each $<\lambda$-presented module $C$. A module is called \emph{$\lambda$-pure-injective} if it is injective with respect to all $\lambda$-pure monomorphisms. The case $\lambda = \aleph_0$ amounts to the classical notions of purity/pure-injectivity.

The following proposition shows in particular that, over a non-right perfect ring, the class of $\aleph_1$-pure-injective modules does not precisely recognize $\aleph_1$-pure monomorphisms, i.e., there is a pure and not $\aleph_1$-pure monomorphism $\nu$ such that $\Hom R\nu N$ is surjective for each $\aleph_1$-pure-injective module $N$.

\begin{proposition} \label{p:enough} Let $R$ be a ring which is not right perfect. Then $\ModR$ does not have enough $\aleph_1$-pure-injective objects.
\end{proposition}

\begin{proof} Let us denote by $\mathcal C$ the class of all $\aleph_1$-pure-injective modules. Since $R$ is not right perfect, there exists a countably presented flat module $M$ which is not projective. This module is the direct limit of a countable directed system $F_0\to F_1\to F_2\to \dotsb$ consisting of finite rank free modules. In particular, we have a pure, albeit not $\aleph_1$-pure, short exact sequence $0\to R^{(\omega)}\overset{\varepsilon}{\to} R^{(\omega)}\to M\to 0$.

Each flat Mittag-Leffler module (being $\aleph_1$-projective, cf.\ \cite[3.19]{GT}) is an $\aleph_1$-pure-epimorphic image of a projective module, whence it belongs to ${}^\perp\mathcal C$. By \cite[Theorem~6]{BS}, $M\in {}^\perp\mathcal C$ as well. It immediately follows that the module $R^{(\omega)}$ cannot be $\aleph_1$-purely embedded into an $\aleph_1$-pure-injective module since this would imply that $\varepsilon$ is $\aleph_1$-pure and thus splits.
\end{proof}

Assuming Axiom of Constructibility (V = L), we can even show that:

\begin{proposition} \label{p:evenmore} (V = L) There does not exist a regular uncountable cardinal $\lambda$ such that $\ModR$ has enough $\lambda$-pure-injective objects, provided that $R$ is non-right perfect.
\end{proposition}

\begin{proof} For the sake of contradiction, assume that there exists such a $\lambda$. Let $f:R^{(\omega)}\to B$ be a $\lambda$-pure embedding into a $\lambda$-pure-injective module $B$. Put $\kappa = (|B|+\lambda)^+$ and fix a countably presented non-projective flat module $N$ together with a pure short exact sequence $0\to R^{(\omega)}\overset{\varepsilon}{\to} R^{(\omega)}\to N\to 0$. By \cite[Theorem~VII.1.4]{EM}, there exists a $\kappa$-free module $M$ which is not projective. In particular, $M$ is a $\kappa$-pure-epimorphic image of a free module, whence $M\in {}^\perp B$. Moreover, by the proof of \cite[Theorem~VII.1.4]{EM}, there are a stationary subset $E$ of $\kappa$ consisting of ordinals with countable cofinality and a filtration $\mathcal M = (M_\alpha\mid\alpha<\kappa)$ of $M$ consisting of $<\kappa$-generated free modules such that $M_{\nu+1}/M_\nu\cong N$ for each $\nu\in E$.

By Lemma~\ref{l:wdiam} (note that $\Phi_\kappa(E)$ holds under V = L), there exists a $\nu\in E$ such that $N\cong M_{\nu+1}/M_\nu\in {}^\perp B$. In particular, $f$ can be factorized through $\varepsilon$ yielding that $\varepsilon$ is $\lambda$-pure, and thus split, in contradiction with the choice of $N$.
\end{proof}

\begin{remark} \label{r:ZFC} If $\lambda = \aleph_n$ for $0<n<\omega$, then the proof of Proposition~\ref{p:evenmore} can be carried out in ZFC using the tools from \cite[Section~VII.3]{EM}.
\end{remark}

\begin{example} If $R$ is a countable ring and $\mathcal C$ the class of all $\aleph_1$-pure-injective modules, then $^\perp\mathcal C$ is just the class of all flat modules (cf.\ \cite[Theorem~6]{BS}). As a~consequence, $\aleph_1$-pure-injective modules over countable rings are cotorsion.

Suppose that $R$ is a countable von Neumann regular ring which is not semisimple. Then all the classes of cotorsion, pure-injective, $\aleph_1$-pure-injective and injective modules coincide. Subsequently, there exists a single monomorphism $\mu$ such that a module $M$ is $\aleph_1$-pure-injective if and only if $\Hom R{\mu}M$ is surjective. On the other hand, there is strictly more pure monomorphisms than $\aleph_1$-pure monomorphisms. Also, since all cyclic modules are countable, it immediately follows from the Baer test lemma that $\aleph_1$-pure submodules of ($\aleph_1$-pure-)injective modules are just the direct summands. Thus not only $\ModR$ does not have enough $\aleph_1$-pure-injective modules, even more so: no non-injective module $\aleph_1$-purely embeds into an $\aleph_1$-pure-injective module.

Consider now the class $\mathcal T$ of all $\aleph_1$-pure embeddings into projective modules. By \cite[Remark~10.7]{GT}, $\{\Coker\nu\mid \nu\in\mathcal T\}$ is precisely the class of all flat Mittag-Leffler modules, hence (again by \cite[Theorem~6]{BS}) $\mathcal T$ is the test class for ($\aleph_1$-pure-)injectivity. However, we cannot pick a subset $\mathcal S\subset \mathcal T$ with the same property, since for each set $\mathcal F$ of flat Mittag-Leffler modules $\mathcal F^\perp$ contains a non-injective module.
\end{example}

\begin{remark}\label{r:FML} We do not know whether there exists a ring $R$ and an $\aleph_1$-pure-injective $R$-module $M$ which is not cotorsion. If it is the case, we would get $M\in\mathcal D^\perp$ where $\mathcal D$ denotes the class of all flat Mittag-Leffler $R$-modules, disproving the conjecture that $\mathcal D^\perp$ is precisely the class of all cotorsion modules.
\end{remark}

\section{Projective modules and their generalizations}
             
We turn to the dual setting of projective modules and their generalizations. These classes of modules are defined by factorization properties with respect to proper classes of epimorphisms. Here, the problem of replacing a proper class of morphisms by a set is more complex: we will see that its solution may depend on the underlying ring, and on the extension of ZFC that we work in. 

Given an epimorphism $\varphi$ in $\rmod R$, we call a class of modules $\mathcal D$ a \emph{cofactorization class} of $\varphi$ provided that $\mathcal D = \{ N \in \rmod R \mid \Hom RN{\varphi} \hbox{ is surjective} \}$.

Our first aim is to dualize Lemma \ref{setgen}. To this purpose, we need more definitions:

Let $\mathcal S$ be a class of modules. A module $M$ is \emph{$\mathcal S$-cofiltered} provided that $M$ is isomorphic to an inverse limit of a well-ordered continuous inverse system $( M _\alpha, \pi_{\alpha \beta} \mid \alpha < \beta < \sigma)$ such that $\pi_{\alpha \beta} : M_\beta \to M_\alpha$ is an epimorphism for all $\alpha < \beta < \sigma$, $M_0 = 0$, and for each $\alpha < \sigma$, $\Ker{\pi_{\alpha \alpha + 1}}$ is isomorphic to an element of $\mathcal S$. The class of all $\mathcal S$-cofiltered modules is denoted by $\hbox{Cofilt}(\mathcal S)$. The class of all direct summands of $\mathcal S$-cofiltered modules is denoted by $\hbox{sCofilt}(\mathcal S)$.

Note that $\hbox{Cofilt}(\hbox{Cofilt}(\mathcal S)) = \hbox{Cofilt}(\mathcal S)$ and $\hbox{Cofilt}(\hbox{sCofilt}(\mathcal S)) = \hbox{Cofilt}(\mathcal S)$. In particular, both $\hbox{Cofilt}(\mathcal S)$ and $\hbox{sCofilt}(\mathcal S)$ are closed under arbitrary direct products and extensions.

A class of modules $\mathcal C$ is \emph{codeconstructible} provided that there is a \emph{subset} $\mathcal S \subseteq \mathcal C$ such that $\mathcal C = \hbox{Cofilt}(\mathcal S)$.     
  
Here is a partial dual of Lemma \ref{setgen}:

\begin{lemma}\label{setcogen} Let $\mathfrak C = (\mathcal A,\mathcal B)$ be a cotorsion pair in $\rmod R$. Consider the following conditions:
\begin{enumerate}
\item $\mathfrak C$ is cogenerated by a set.
\item There is a subset $\mathcal S \subseteq \mathcal B$ such that $\mathcal B = \hbox{sCofilt}(\mathcal S)$.
\item $\mathcal B$ is codeconstructible.
\item There is an epimorphism $g \in \Hom RIJ$ with $I$ injective, such that $\mathcal A$ is the cofactorization class of $g$.
\end{enumerate}
Then (1) is equivalent to (4). Moreover, (3) implies (2), and (2) implies (1).  
\end{lemma}
\begin{proof} The proof of the equivalence of (1) and (4) is dual to the proof of the corresponding equivalence in Lemma \ref{setgen}. 
Moreover, (3) implies (2) because $\mathcal B$ is closed under direct summands. That (2) implies (1) follows by the Lukas Lemma \cite[6.37]{GT}.
\end{proof} 

\medskip
The following example shows that (2) need not imply (3) even for the trivial cotorsion pair $\mathfrak I = (\rmod R,\mathcal I_0)$, i.e., 
there is no dual to Kaplansky's theorem for cotorsion pairs for $\mathfrak I$. 

\begin{example}\label{codeconstr1} (i) For any ring $R$, $\mathfrak I$ is cogenerated by any set consisting of injective modules. Consider such a set $\mathcal S$. Then $\hbox{Cofilt}(\mathcal S)$ is the class of all modules isomorphic to (arbitrary) direct products of modules from $\mathcal S$. In particular, if $\mathcal S _0 = \{ W \}$ where $W$ is an injective cogenerator for $\rmod R$, then $\mathcal I_0 = \hbox{sCofilt}(\mathcal S _0)$. So condition (2) is satisfied.
 
(ii) Now, assume that $R$ is a commutative noetherian ring with a maximal ideal $m$ satisfying $\bigcap_{n < \omega} m^n = 0$ and $m \notin \hbox{Ass} R$ (this occurs when $R$ is a noetherian domain which is not a field, or $R$ is a local Gorenstein ring of Krull dimension $\geq 1$). Assume that $\mathcal I _0 = \hbox{Cofilt}(\mathcal S)$ for a set $\mathcal S \subseteq \mathcal I _0$. Then $\mathcal S$ must contain copies of all indecomposable injectives. Since $\bigcap_{n < \omega} m^n = 0$, if $\mathcal D$ is any infinite sequence of elements of $\mathcal S$ such that for each $D \in \mathcal D$, $m \in \Ass D$, then $P = \prod_{D \in \mathcal D} D$ contains a copy of $R$. So $P$ has indecomposable direct summands with associated primes different from $m$. It follows that for each infinite cardinal $\kappa$, $E(R/m)^{(\kappa )}$ is isomorphic to a~module in $\mathcal S$, in contradiction with $\mathcal S$ being a set. Thus $\hbox{Cofilt}(\mathcal S) \subsetneq \mathcal I _0$ for any set $\mathcal S$ of injective modules, i.e., condition (3) fails.  
\end{example} 

Our next example generalizes part (i) of Example \ref{codeconstr1} to $1$-cotilting cotorsion pairs. Hence, it also covers the setting of Dedekind domains and cotorsion pairs $(\mathcal F, \mathcal C)$ such that $\mathcal F \supseteq \mathcal F _0$ (see Theorem \ref{Dedekind} below): 

\begin{example}\label{codeconstr2} Let $R$ be a ring and $\mathfrak C = (\mathcal A, \mathcal B)$ be a cotilting cotorsion pair cogenerated by a $1$-cotilting module $C$, i.e., a module $C$ such that $\Cog{C}={}^\perp C$. Then $\mathfrak C$ satisfies the condition (2) of Lemma~\ref{setcogen}. 

Indeed, using \cite[15.20]{GT} and a dual version of \cite[6.13]{GT}, we see that $\mathcal B$ coincides with the class of all modules that are direct summands of the modules $M$ that fit in an exact sequence $0 \to C^\lambda \to M \to I \to 0$ for an injective module $I$ and a cardinal $\lambda$. Let $W$ denote an injective cogenerator for $\rmod R$. Then $I \oplus J \cong W^\kappa$ for a cardinal $\kappa$, and we have the exact sequence $0 \to C^\lambda \to M \oplus J \to W^\kappa \to 0$. Let $\mathcal S = \{ C, W \}$ ($\subseteq \mathcal B$). Then the module $M \oplus J$ is $\mathcal S$-cofiltered, whence $\mathcal B = \hbox{sCofilt}(\mathcal S)$.    
\end{example}
  
Lemma \ref{setcogen} is sufficient to settle our problem for flat and torsion-free modules:

\begin{corollary}\label{flatetc} Let $R$ be a ring and $n < \omega$. Then there exist epimorphisms $g_n \in \Hom R{I_n}{J_n}$ and $h \in \Hom RIJ$ with $I_n$ and $I$ injective, such that the class $\mathcal F _n$ of all modules of flat dimension $\leq n$ (the class of all torsion free modules) is a cofactorization class of $g_n$ ($h$).
\end{corollary}
\begin{proof} The first claim follows by Lemma \ref{setcogen}, since a module $M$ has weak dimension $\leq n$, iff $\Tor {n+1}RMC = 0$ for each cyclic left $R$-module $C$, iff $\Tor 1RM{\Omega ^n(C)} = 0$ for each cyclic left $R$-module $C$, iff $\Ext 1RMD = 0$ for each module $D$ which is the dual module of the $n$th syzygy left $R$-module $\Omega ^n(C)$ of a cyclic left $R$-module~$C$. For the second claim, one uses the characterization of torsion-free modules as the modules $M$ such that $\Tor 1RM{R/Rr} = 0$ for each non-zero divisor $r \in R$.  
\end{proof}

In particular, we have a solution for modules of bounded projective dimension in the case when projectivity and flatness of modules coincide, that is, for right perfect rings:
 
\begin{corollary}\label{perfect} Let $R$ be a right prefect ring and $n < \omega$. Then there exists an epimorphism $g_n \in \Hom R{I_n}{J_n}$ with $I_n$ injective, such that the class $\mathcal P _n$ of all modules of projective dimension $\leq n$ is a cofactorization class of $g_n$.
\end{corollary}

\begin{remark}\label{variants} Note that, over each right perfect ring $R$, there are simpler test epimorphisms than $g_0$ from Corollary \ref{flatetc} available: 

1. By \cite{S} (or \cite{KV}), the Dual Baer Criterion (DBC) for projectivity holds in $\rmod R$, that is, one can use the epimorphism $g : R^{T} \to \prod_{I \in T} I$, where $T$ is the set of all right ideals of $R$: $\mathcal P _0$ is the cofactorization class of $g$.

2. Let $\pi$ be the finite direct product of the projections $\pi_S : E(S) \to E(S)/S$ where $S$ runs over a representative set of all simple modules. Then $\mathcal P _0$ is the cofactorization class of $\pi$ by \cite[8.8]{GT}. 
\end{remark}

For non-right perfect rings, the following consistency result gives a barrier to testing for projectivity using \emph{sets} of epimorphisms: 

\begin{lemma}\label{sup} Assume Shelah's Uniformization Principle (SUP). Let $R$ a non-right perfect ring. 

Then for each set $\mathcal S\subset\rmod{R}$, there exists a module $M$ with $\pd M = 1$ and $\Ext 1RMS = 0$ for all $S \in \mathcal S$. 
So the cotorsion pair $(\mathcal P_0,\rmod R)$ is not cogenerated by any set of modules. 

Moreover, there is \emph{no} epimorphism $\pi$ such that $\mathcal P _0$ is the cofactorization class of~$\pi$.     
\end{lemma}
\begin{proof} For the first claim, see e.g.\ \cite[2.4]{T1}. The second claim follows from the first one for $\mathcal S = \{ \Ker{\pi} \}$. 
\end{proof}  

Note that if ZFC is consistent, then so is ZFC + SUP, see \cite{ES}. 

\bigskip
We now pause to look at other cotorsion pairs, but in the particular setting of Dedekind domains. Once again, (SUP) comes as a handy tool here. The following theorem generalizes \cite[1.3]{EST}: 

\begin{theorem}\label{Dedekind} Let $R$ be a Dedekind domain with the quotient field $Q$, and $\mathfrak C = (\mathcal F, \mathcal C)$ be a cotorsion pair. 
\begin{enumerate}
	\item Assume $\mathcal F _0 \subseteq \mathcal F$. Then $\mathfrak C$ is a cotilting cotorsion pair, i.e., $\mathfrak C$ is cogenerated by a pure-injective module $C$ such that $\Cog C = {}^\perp C \, (= \mathcal F)$. In particular, $\mathfrak C$ is generated by a set, and there is a unique subset $P \subseteq \mSpec R$ such that $\mathcal F = \{ M \in \rmod R \mid  \forall p \in P : \Hom R{R/p}M = 0 \}$.   
	\item Assume (SUP). Moreover, assume that $\mathcal F _0 \nsubseteq \mathcal F$, $Q$ is a countably generated module, and $\mathfrak C$ is generated by a set.
		
Then there is \emph{no} epimorphism $\pi$ such that $\mathcal F$ is the cofactorization class of $\pi$.           
\end{enumerate} 
\end{theorem}
\begin{proof} (1) This is proved in \cite[16.29]{GT} (see also \cite[16.21]{GT}). 

(2) In \cite[1.3]{EST}, it is proved (in a different notation) that under the same assumptions, $\mathfrak C$ is not cogenerated by a set, i.e., there is no epimorphism $\pi : I \to J$ \emph{with $I$ injective} such that $\mathcal F$ is the cofactorization class of $\pi$ (see Lemma \ref{setcogen}). In order to prove our stronger claim, we will follow the pattern of the proof of \cite[1.3]{EST}, indicating only the necessary changes needed to strengthen the claim. 

Assume there is an epimorphism $\pi$ in $\rmod R$ such that $\mathcal F$ is the cofactorization class of $\pi$. Let $K = \Ker{\pi}$.

Since $Q$ is countably generated, there is a  short exact sequence $0 \to R^{(\omega)} \overset{f}\to R^{(\omega)} \to Q \to 0$ with $f (1_i) = 1_i - \rho_i 1_{i+1}$ where $\{ \rho_i \mid i < \omega \}$ is  a set of non-invertible non-zero elements of $R$. As in \cite{EST}, we use this presentation and (SUP) to define for a cardinal $\mu = \tau ^+$, where $\tau$ is an uncountable singular cardinal of cofinality $\omega$, a~non-projective module $H = F/G$, where $F$ and $G$ are free modules of rank $\mu$, such that $\Ext 1RHN = 0$ for each module $N$ of cardinality $< \tau$.        

As in \cite{EST}, we distinguish two cases: (1) $\mathfrak C$ is generated by a cotorsion module $B$, and (2) there is no cotorsion module generating $\mathfrak C$ (but, by our assumption, there is still a non-cotorsion module $B$ such that $\mathcal C = B^\perp$).  

In the case (2), we take a singular cardinal $\tau$ of cofinality $\omega$ bigger than the cardinalities of $R$, $B$ and $K$ and $\mu = \tau^+$. Taking the appropriate $\{ B \}$-filtered module $A \in B^\perp = \mathcal C$, we infer as in the proof of case (2) in \cite[1.3]{EST} that for the module $H$ as above, $\Ext 1RHK = 0$, so $H \in \mathcal F$, but $\Ext 1RHA \neq 0$, in contradiction with $A \in \mathcal C$.   

In the case (1), we will first prove the following claim: $\mathcal C = B^\perp = (E(T(B)) \oplus F(B))^\perp$ where $T(B)$ is the torsion part of $B$, $E(T(B))$ is its injective hull, and $F(B)$ is the torsion-free part of $B$. 

Consider the push-out of the embeddings $T(B) \subseteq B$ and $T(B) \subseteq E(T(B))$: 

$$\begin{CD}
@.     0@.        0      @.     @.\\
@.     @VVV       @VVV   @.     @.\\
0@>>>  T(B)@>{\subseteq}>> B@>>>  F(B)@>>>  0\\
@.     @V{\subseteq}VV  @VVV   @|     @.\\
0@>>>  E(T(B))@>>>      X@>>>  F(B)@>>>  0\\
@.     @VVV       @VVV @.  @.\\
@.     E(T(B))/T(B) @=      E(T(B))/T(B)      @.     @.\\
@.     @VVV       @VVV   @.     @.\\
@.     0@.        0.      @.     @.
\end{CD}$$
 
Since the second row splits, it suffices to prove that $B^\perp = X^\perp$. As $R$ is hereditary, the second column gives $X^\perp \subseteq B^\perp$, and the first row yields $B^\perp \subseteq T(B) ^\perp$. Let $P$ be the set of all maximal ideals $p$ such that $p \in \Ass T(B)$. Then $E(T(B))^\perp = T(B)^\perp = (\bigoplus_{p \in P} R/p)^\perp = \{ M \in \rmod R \mid \forall p\in P : M.p = M \}$, since $E(T(B))$ is $\{ R/p \mid p \in P \}$-filtered. Moreover, there is a subset $P^\prime \subseteq P$ such that $E(T(B))/T(B)$ is $\{ R/p \mid p \in P^\prime \}$-filtered. Hence $T(B)^\perp \subseteq  (E(T(B))/\hbox{T(B)})^\perp$. The second column gives $B^\perp \cap (E(T(B))/T(B))^\perp \subseteq X^\perp$. However, $B^\perp \subseteq T(B) ^\perp \subseteq (E(T(B))/\hbox{T(B)})^\perp$, whence $B^\perp \subseteq X^\perp$, and our claim is proved. 

Furthermore, $F(B)$ is a torsion-free cotorsion module, which is reduced (because $\mathcal F _0 \nsubseteq \mathcal F$), so it is isomorphic to a product of completions of localizations of free $R_q$-modules where $q$ ranges over some subset $S \subseteq \mSpec R$, see \cite[5.3.28]{EJ}. If $q \in S \cap P$, then the $q$-adic module $J_q \in \mathcal F$ as well as $E(R/q) \in \mathcal F$. Since $\mathcal F$ is closed under extensions, also $E(J_q) \in \mathcal F$, whence $Q \in \mathcal F$, a contradiction. Thus, $S \cap P = \emptyset$. Since the completion of the localizations of a free $R_q$-module is $p$-divisible for all $p \neq q$, we infer that $F(B).p = F(B)$ for all $p \in P$. 

The rest of the proof is the same as for the case (1) in \cite[1.3]{EST}, i.e., we take a~singular cardinal $\tau$ of cofinality $\omega$ bigger than the cardinalities of $R$, $B$ and $K$ and let $\mu = \tau^+$, and take the appropriate $\{ F(B) \}$-filtered module $A \in B^\perp = \mathcal C$ so that for the module $H$ as above, $\Ext 1RHK = 0$, whence $H \in \mathcal F$, but $\Ext 1RHA \neq 0$, in contradiction with $A \in \mathcal C$.          
\end{proof}

\medskip

Next we show that in the extensions of ZFC with the Weak Diamond Principle~$\Phi$, for many non-right perfect rings, the class of all projective modules is the cofactorization class of a suitable (single) epimorphism. The point is that one can combine the methods developed (in ZFC) in infinite dimensional tilting theory \cite[\S 8.2]{GT} with the following lemma employing $\Phi$ which generalizes \cite[Lemma~A.7]{FS0} and \cite[Theorem~XII.1.10]{EM}.

\smallskip

First, however, we have to recall the statement of the principle $\Phi$. Given a set $X$, we say that a system $(X_\alpha\mid\alpha<\kappa)$ is a \emph{$\kappa$-filtration of $X$} if $X = \bigcup_{\alpha<\kappa}X_\alpha$, $|X_\alpha|<\kappa$, $X_\alpha\subseteq X_{\alpha+1}$ for each $\alpha<\kappa$, and $X_\alpha = \bigcup_{\beta<\alpha}X_\beta$ whenever $\alpha<\kappa$ is a~limit ordinal.

Given a regular uncountable cardinal $\kappa$ and a stationary subset $S\subseteq \kappa$, we denote by $\Phi_\kappa(S)$ the following statement: let $X$ be a set with $|X|\leq\kappa$ and $(X_\alpha\mid \alpha\leq\kappa)$ be a $\kappa$-filtration of $X$; moreover, for each $\alpha\in S$ let a mapping $P_\alpha:\mathcal P(X_\alpha)\to \{0,1\}$ be given. Then there exists $\varphi:S\to \{0,1\}$ such that, for each $Y\subseteq X$, the set $\{\alpha\in S\mid P_\alpha(Y\cap X_\alpha) = \varphi(\alpha)\}$ is stationary in $\kappa$.

The Weak Diamond Principle $\Phi$ asserts that ``$\Phi_\kappa(S)$ holds for each regular uncountable cardinal $\kappa$ and stationary subset $S\subseteq\kappa$''. It is consistent with ZFC since it follows, for instance, from the axiom of constructibility V = L. It is however much weaker than V = L, and also weaker than the famous Jensen's diamond principle~$\diamondsuit$.

\begin{lemma}\label{l:wdiam} Let $A, B$ be $R$-modules with $|B|\leq \kappa$ where $\kappa$ is a regular uncountable cardinal. Assume that $A$ is the direct limit of a continuous well-ordered directed system $\mathcal A = (A_\alpha,h_{\beta \alpha}:A_\alpha\to A_\beta \mid \alpha<\beta<\kappa)$ where $A_\alpha$ is $<\kappa$-generated and $\Ext 1R{A_\alpha}B = 0$ for each $\alpha<\kappa$. If $S \subseteq \{\alpha<\kappa\mid \Ext 1R{\Coker{h_{\alpha+1 \alpha}}}B\neq 0\}$ is stationary in $\kappa$ and $\Phi_\kappa(S)$ holds, then $\Ext 1RAB\neq 0$.
\end{lemma}

\begin{proof} First, we extend the system $\mathcal A$ into a continuous well-ordered system of short exact sequences $\mathcal E_\alpha:0\to K_\alpha\overset{\subseteq}{\to} F_\alpha\overset{\pi_\alpha}{\to} A_\alpha\to 0$ where $F_\alpha$ is a free module of rank $<\kappa$ and the three components of connecting maps $\varepsilon_{\alpha+1 \alpha}:\mathcal E_\alpha\to \mathcal E_{\alpha+1}$ are inclusion, split inclusion and $h_{\alpha+1 \alpha}$, respectively. Using the assumptions that $\kappa$ is infinite regular and $A_\alpha$ is $<\kappa$-generated for $\alpha<\kappa$, this is easy.

As the direct limit, we get a short exact sequence $0\to K \overset{\subseteq}{\to} F\to A\to 0$ where $F$ is free of rank $\leq\kappa$. Thus we obtain also a $\kappa$-filtration $(V_\alpha\mid \alpha<\kappa)$ of a set $V$ of free generators of $F$ where $V_\alpha$ is a set of free generators of $F_\alpha$ for each $\alpha<\kappa$.

Since $\Ext 1R{A_\alpha}B = 0$ for $\alpha<\kappa$, we can fix for each homomorphism $f:K_\alpha \to B$ one of its extensions $f^e\in\Hom R{F_\alpha}B$. Furthermore, we fix, for each $\alpha\in S$, a~$k_\alpha\in\Hom R{A_\alpha}B$ which cannot be factorized through $h_{\alpha+1\alpha}$.

Consider any $\kappa$-filtration $(B_\alpha\mid\alpha<\kappa)$ of the set $B$ (the $B_\alpha$ need not be submodules of $B$, just subsets) and put $X = V\times B$ and $X_\alpha = V_\alpha\times B_\alpha$ for each $\alpha<\kappa$. Let $\alpha\in S$ be arbitrary. We define the mapping $P_\alpha:\mathcal P_\alpha\to \{0,1\}$ as follows: if $Z\subseteq X_\alpha$ is not a mapping from $V_\alpha$ to $B_\alpha$, we put $P_\alpha(Z) = 0$; otherwise, we fix a unique extension $z\in\Hom R{F_\alpha}B$ of $Z$ and put $y = (z\restriction K_\alpha)^e$. Then $y-z$ is zero on $K_\alpha$ and thus it defines a unique homomorphism from $A_\alpha$ to $B$. We put $P_\alpha(Z) = 1$ if and only if this homomorphism can be factorized through $h_{\alpha+1\alpha}$.

Using $\Phi_\kappa(S)$, we get $\varphi:S\to \{0,1\}$ for our choice of the mappings $P_\alpha$. To show that $\Ext 1RAB\neq 0$, we recursively construct a homomorphism $f:K\to B$ which cannot be extended to an element of $\Hom RFB$. We start with $f_0:K_0 \to B$ the zero map. If $\alpha\leq\kappa$ is a limit ordinal, we put $f_\alpha = \bigcup_{\beta<\alpha}f_\beta$. Let us assume that $f_\alpha$ is already constructed and $\alpha<\kappa$. We define $f_{\alpha+1}: K_{\alpha+1}\to B$ as follows:

Let us put $f_\alpha^\prime = f_\alpha^e$ if $\alpha\not\in S$ or $\varphi(\alpha) = 0$; otherwise, we put $f_\alpha^\prime = f_\alpha^e+k_\alpha\pi_\alpha$. We extend $f_\alpha^\prime$ arbitrarily to a homomorphism $f_\alpha^+:F_{\alpha+1}\to B$ and define $f_{\alpha+1}$ as $f_\alpha^+\restriction K_{\alpha+1}$.

Finally, we put $f = f_\kappa:K\to B$. For the sake of contradiction, we assume that there exists $g\in\Hom RFB$ such that $g\restriction K = f$. It is easy to see that the set $C = \{\alpha<\kappa\mid g(V_\alpha)\subseteq B_\alpha\}$ is closed and unbounded. Using the property of $\varphi$ for $Y = g\restriction V$, we obtain a $\delta\in C\cap S$ such that $P_\delta(g\restriction V_\delta) = \varphi(\delta)$. Obviously, $f_\delta^+-g\restriction F_{\delta+1}$ is zero on $K_{\delta+1}$. Thus, it defines a unique $h\in\Hom R{A_{\delta+1}}B$. This $h$ extends the homomorphism $k:A_\delta\to B$ where $k\pi_\delta = f_\delta^\prime-g\restriction F_\delta$.

If $\varphi(\delta) = 0$, then $k\pi_\delta = f_\delta^e-g\restriction F_\delta$ in the contradiction with $P_\delta(g\restriction V_\delta) = 0$ which has meant that $k$ cannot be factorized through $h_{\delta+1\delta}$.

If on the other hand $\varphi(\delta) = 1$, then $k\pi_\delta = k_\delta\pi_\delta + f_\delta^e - g\restriction F_\delta$. Since $P_\delta(g\restriction V_\delta) = 1$, we know that $k-k_\delta$ can be factorized through $h_{\delta+1\delta}$ which immediately implies that $k_\delta$ has this property as well, in contradiction with its choice.
\end{proof}

A typical application of Lemma~\ref{l:wdiam} is the following: given a filtration $\mathcal A = (A_\alpha\mid \alpha<\kappa)$ of $A$ consisting of $<\kappa$-generated modules such that $A, A_\alpha\in {}^\perp B$ for each $\alpha<\kappa$ and $|B|\leq\kappa$, there exists a subfiltration of $\mathcal A$ with consecutive factors in ${}^\perp B$ provided that $\Phi$ holds true.

In the sequel, given a $\mathcal C\subseteq\ModR$ and a cardinal $\mu$, we denote by $\mathcal C^{\leq\mu}$ (or $\mathcal C^{<\mu}$, respectively) the subclass of $\mathcal C$ consisting of all the modules which are $\mu$-presented ($<\mu$-presented, respectively).

Also, we will say that a module $M$ has \emph{$\mathcal C$-resolution dimension} $\leq n$ provided that there exists an exact sequence 
$\mathcal E : 0\to C_n\to C_{n-1}\to\dotsb\to C_1\to C_0\to M\to 0$ with $C_i \in \mathcal C$ for all $i \leq n$. The sequence $\mathcal E$ is called a \emph{$\mathcal C$-resolution} of $M$ of length $n$.

\begin{theorem}\label{boundedpd} Assume that $\Phi$ holds true. Let $\mathfrak C = (\mathcal A,\mathcal B)$ be a~hereditary cotorsion pair such that $\mathcal A\cap\mathcal B = \Add K$ for some $K\in\ModR$. Assume that $\mu \geq |K|+|R|$ is an infinite cardinal such that $\mathcal A = \hbox{Filt}(\mathcal A^{\leq\mu})$. 

Let $M$ be a module of $\mathcal A$-resolution dimension $\leq n$ (e.g., assume $M \in \mathcal P _n$). Then $M\in\mathcal A$ if and only if $\Ext iRM{K^{(\mu)}} = 0$ for all $0 < i \leq n$. 
\end{theorem}

\begin{proof} The only-if part follows immediately since $K^{(\mu)}\in\mathcal B$ and $\mathfrak C$ is hereditary. Let us concentrate on the if part. We shall prove it by induction on $n$. The result is trivial for $n = 0$, so let $n = 1$.

We thus have a short exact sequence $\mathcal E:0\to A_1\to A_0\overset{\pi}{\to} M\to 0$ with $A_0,A_1\in\mathcal A$. Considering a special $\mathcal B$-preenvelope of $A_1$ and subsequently forming the obvious push-out, we can w.l.o.g.\ assume that $A_1\in \Add K$. By further adding a~suitable direct summand from $\Add K$ to $A_0$ and $A_1$, we can moreover assume that $A_1 = K^{(\kappa)}$ for some cardinal $\kappa$. We argue by induction on $\kappa$ that $M\in\mathcal A$.

If $\kappa\leq\mu$, then the short exact sequence $\mathcal E$ splits by our assumption, whence $M\in\mathcal A$. Assume that $\kappa = |A_1|>\mu$. Then $A_1$ is trivially $\{K\}$-filtered and $A_0$ is $\mathcal A^{\leq\mu}$-filtered by one of our assumptions. If $\kappa$ is regular, we use \cite[Theorem~3.4]{ST} to obtain a filtration $(\mathcal E_\alpha:0\to A_{\alpha,1}\to A_{\alpha,0}\to M_\alpha\to 0\mid \alpha<\kappa)$ of $\mathcal E$ where $A_{\alpha,1}$ is a~canonical direct summand of $A_1$ and $A_{\alpha,0}\in\mathcal A$, whilst $|A_{\alpha,0}|<\kappa$, for each $\alpha<\kappa$. Consequently, since $\Hom R{-}{K^{(\mu)}}$ turns $\mathcal E$ into a short exact sequence by our assumption, we see that $\Ext 1R{M_\alpha}{K^{(\mu)}} = 0$ for each $\alpha<\kappa$. Applying Lemma~\ref{l:wdiam}, we easily obtain a subfiltration of $(M_\alpha\mid \alpha<\kappa)$ with consecutive factors in ${}^\perp K^{(\mu)}$. Using the property of $M_{\alpha+1}/M_\alpha$ guaranteed by \cite[Theorem~3.4]{ST}, we see that thus obtained consecutive factors belong to $\mathcal A$ by inductive assumption, whence $M\in\mathcal A$ by the Eklof lemma.

If $\kappa>\mu$ is singular, we easily construct, for each regular $\lambda<\kappa$ such that $\mu<\lambda$, a system $\mathcal S_\lambda$ consisting of subobjects of the short exact sequence $\mathcal E$ whose first term is a canonical direct summand of $A_1$ of cardinality $<\lambda$, the second term is in $\mathcal A^{<\lambda}$, and such that $\bigcup\mathcal S_\lambda = \mathcal E$ and $\mathcal S_\lambda$ is closed under unions of chains of cardinality $<\lambda$. The third terms of short exact sequences in $\mathcal S_\lambda$ then belong to ${}^\perp K^{(\mu)}$, and thus to $\mathcal A = \hbox{Filt}(\mathcal A^{\leq\mu})$ by the inductive hypothesis. We use the Singular Compactness Theorem, \cite[Theorem~7.29]{GT}, to deduce that $M\in\mathcal A$. This settles the case $n = 1$.

Finally, assume that $n>1$. Then the first syzygy, $\Omega(M)$, in an $\mathcal A$-resolution of $M$ of length $n$ has $\mathcal A$-resolution dimension $\leq n-1$. By the inductive hypothesis, $\Omega(M) \in\mathcal A$ which immediately implies that $M\in\mathcal A$, too.
\end{proof}

\begin{remark}\label{higherpd} Dimension shifting yields a variant of Theorem \ref{boundedpd} for higher $\mathcal A$-resolution dimensions: if $0 \leq d < n$, then $M$ has $\mathcal A$-resolution dimension $\leq d$, iff $\Ext iRM{K^{(\mu)}} = 0$ for all $d < i \leq n$.   

Also, it follows from Theorem \ref{boundedpd} that if $M$ is a module of $\Add K$-resolution dimension $\leq n$, then $M \in \Add K$, iff 
$\Ext iRM{K^{(\mu)}} = 0$ for all $0 < i \leq n$.

All tilting cotorsion pairs, including, of course, the trivial one $(\mathcal P_0,\ModR)$, satisfy the assumptions of Theorem~\ref{boundedpd}. More generally, all hereditary cotorsion pairs with the right-hand class closed under direct limits (cf.\ \cite[Lemma~5.4]{AST}). Other examples include, for instance, the cotorsion pair $\mathfrak{PGF} = (\mathcal{PGF},\mathcal{PGF}^\perp)$ from \cite{SS}. In particular, if $R$ is a ring over which each module has finite $\mathcal{PGF}$-resolution dimension, then $\Phi$ implies the existence of a test module for projectivity in $\ModR$: indeed, by Theorem~\ref{boundedpd}, $\mathfrak{PGF}$ is cogenerated by a~single module. At the same time, the flat cotorsion pair $(\mathcal F_0,\mathcal{EC})$ is also cogenerated by a~single module, and the flat modules in $\mathcal{PGF}$ are precisely the projective ones.
\end{remark}

As the following proposition shows, the Weak Diamond Principle can help us also in another special case which covers the setting of Theorem~\ref{Dedekind} and complements part $(2)$ therein.

\begin{proposition}\label{p:I1} Let $\mathfrak C = (\mathcal A,\mathcal B)$ be a cotorsion pair with $\mathcal B\subseteq\mathcal I_1$. Assuming $\Phi$, $\mathfrak C$ is generated by a set if and only if $\mathfrak C$ is cogenerated by a set.
\end{proposition}

\begin{proof} The if part follows e.g.\ by \cite[Theorem~11.2]{GT} where $\Phi$ and our Lemma~\ref{l:wdiam} play the role of $\Psi$ and \cite[Lemma~11.1]{GT}. Let us prove the only-if part.

Denote by $\kappa$ the least infinite cardinal such that $\kappa\geq |R|$ and $\mathcal A = \hbox{Filt}(\mathcal A^{\leq\kappa})$. For each infinite cardinal $\mu$, fix a module $W_\mu\in\mathcal B$ such that a $<\mu$-presented module $M$ belongs to $\mathcal A$ if and only if $M\in {}^\perp W_\mu$. The module $W_\mu$ can be defined e.g.\ as $\prod_{M\in \mathcal S} W_M$ where $\mathcal S$ is the representative set of all $<\mu$-presented modules not belonging to $\mathcal A$ and $W_M\in\mathcal B$ is such that $\Ext 1RM{W_M} \neq 0$.

We construct a continuous increasing sequence $(\lambda_\alpha\mid \alpha\leq\kappa^+)$ of cardinals: we start by putting $\lambda_0 = \kappa^+$; if $\alpha$ is limit, we put $\lambda_\alpha = \sum_{\beta<\alpha}\lambda_\beta$. Finally, if $\lambda_\alpha$ is defined and $\alpha<\kappa^+$, we set $\lambda_{\alpha+1} = (\lambda_\alpha + |W_{\lambda_\alpha}|)^\kappa$. Put $\lambda = \lambda_{\kappa^+}$. By the definition, we see that $\lambda^\kappa = \lambda$.

Let $W$ denote the submodule of $\prod_{\alpha<\kappa^+}W_{\lambda_\alpha}$ consisting of elements of bounded support (i.e., of support of cardinality $\leq\kappa$). Then $W\in\mathcal B$ since $\mathcal A = \hbox{Filt}(\mathcal A^{\leq\kappa})$. Also $|W| = \lambda^\kappa = \lambda$. We claim that $\mathcal A = {}^\perp W$. Clearly, $\mathcal A\subseteq {}^\perp W$. Let $M\in {}^\perp W$ be arbitrary. We shall prove that $M\in\mathcal A$ by induction on $\theta = |M|$.

If $\theta<\lambda$, then $M\in\mathcal A$ immediately follows by the definition of modules $W_{\lambda_\alpha}$, $\alpha<\kappa^+$. Assume that $N\in {}^\perp W \Rightarrow N\in\mathcal A$ holds for all modules $N$ of cardinality $<\theta$ for some $\theta\geq\lambda$.

First, consider the case of regular $\theta = |M|$ where $M\in {}^\perp W$. Since $W\in\mathcal I_1$ by our assumption, ${}^\perp W$ is closed under submodules, and we can use Lemma~\ref{l:wdiam} to infer that there is a ${}^\perp W$-filtration of $M$ consisting of modules of cardinality $<\theta$. By the inductive assumption, this is actually an $\mathcal A$-filtration, whence $M\in\mathcal A$ by the Eklof lemma.

If $\theta$ is singular, we use the inductive hypothesis and $W\in\mathcal I_1$ to deduce that, for each regular $\mu<\theta$ with $\kappa<\mu$, the system $\mathcal S_\mu$ of submodules in $M$ of cardinality $<\mu$ consists of modules from $\mathcal A = \hbox{Filt}(\mathcal A^{\leq\kappa})$. Consequently, we can use Singular Compactness Theorem \cite[Theorem~7.29]{GT}, to infer that $M\in\mathcal A$.
\end{proof}

Combining the results above, we obtain
      
\begin{corollary}\label{flatbounded} Let $R$ be a non-right perfect ring such that each flat module has finite projective dimension. Then the assertion {\lq \lq}There exists an epimorphism $\pi$ such that $\mathcal P _0$ is the cofactorization class of $\pi${\rq \rq} is independent of ZFC. 

In particular, this is the case when 
\begin{itemize}
\item $R$ is a commutative noetherian ring of non-zero finite Krull dimension,
\item $R$ is an $n$-Iwanaga-Gorenstein ring for $n \geq 1$,
\item $R$ is an almost perfect ring.
\end{itemize}
\end{corollary}
\begin{proof} Consistency of the failure of the assertion under SUP follows by Lemma \ref{sup}. 

Assuming $\Phi$, we can apply Theorem \ref{boundedpd} for $K = R$ and $\mathfrak C = (\mathcal P_0,\ModR)$ to obtain the epimorphism $\pi : \bigoplus_{0 < i \leq n} E(\Omega^{-i}(R^{(R)})) \to E(\Omega^{-i}(R^{(R)}))/\Omega^{-i}(R^{(R)})$, where $\Omega^{-i}(R^{(R)})$ denotes the $i$th cosyzygy of $R^{(R)}$,  which tests for projectivity of modules of finite projective dimension. From our assumption that all flat modules have finite projective dimension it follows that $\mathcal P _0$ is the cofactorization class of the epimorphism $\pi \oplus g_0$, where $g_0$ is the epimorphism from Corollary \ref{flatetc}.      

Finally, flat modules over a commutative noetherian ring of Krull dimension $d < \omega$ have projective dimension $\leq d$ by a classic result of Gruson and Jensen (cf.\ \cite[4.2.8]{X}). If $R$ is $n$-Iwanaga-Gorenstein, then $\pd M \leq n$ for each module of finite flat dimension by \cite[9.1.10]{EJ}. If $R$ is almost perfect, then all flat modules have projective dimension $\leq 1$ by \cite[7.1]{FS} (see also \cite{FN}). 
\end{proof}

\begin{remark}\label{dualbaer} By Remark \ref{variants}.1, the DBC holds for each right perfect ring, that is, $g : R^{T} \to \prod_{I \in T} I$ where $T$ is the set of all right ideals of $R$, is a test epimorphism for projectivity.

By Lemma \ref{sup}, it is consistent with ZFC that the DBC fails for each non-right perfect ring. In fact, in contrast with Corollary \ref{flatbounded}(2), DBC fails (in ZFC) for all commutative noetherian rings that are not perfect (i.e., have Krull dimension $\geq 1$) by \cite{H}. However, there do exist commutative hereditary rings for which DBC is independent of ZFC, \cite{T2}.
\end{remark}

\section{Projective modules and large cardinals}

The question of whether it is consistent with ZFC that there exists a test epimorphism for projectivity over any ring remains open. Apart from using the Weak Diamond Principle, another possible approach here is to utilize large cardinals; in particular the strongly compact ones. Recall that an uncountable cardinal $\kappa$ is called \emph{strongly compact} provided that each $\kappa$-complete filter (on any set $I$) can be extended to a $\kappa$-complete ultrafilter.

In \cite{Sa}, it is shown that, assuming the existence of a strongly compact $\kappa > |R|$, there exists a free $R$-module $F$ such that all projective modules belong to $\Prod(F)$. It is straightforward to check then that a module $M$ of finite projective dimension such that $\Ext iRMF = 0$ for all $i>0$ is necessarily projective. As a result, in the proof above, one can alternatively assume that there exists a proper class of strongly compact cardinals, instead of $\Phi$. We should note, however, that consistency results not relying on large cardinals are usually considered more valuable.

Nonetheless, using strongly compact cardinals, we can get to a situation which resembles the case of right perfect rings where a test module for projectivity always exists. We need just one more preparatory lemma inspired by \cite[Proposition~2.31]{AR}. In its proof, we use the characterization of $\lambda$-pure monomorphisms via solvability of systems of $R$-linear equations, cf.\ \cite[Exercise~XIII.7.2]{FS0}. Also recall that, given a~filter $\mathcal F$ on $I$, \emph{a reduced power} $B^I/\mathcal F$ is defined as $B^I/Z_\mathcal F$ where $Z_\mathcal F = \{f\in B^I\mid \{i\in I\mid f(i) = 0\}\in\mathcal F\}$. If $\mathcal F$ is actually an ultrafilter, then $B^I/\mathcal F$ is called \emph{an ultrapower} of $B$.

\begin{lemma}\label{l:pureepi} Let $\lambda$ be an infinite regular cardinal, $f:A\to B$ a $\lambda$-pure monomorphism which is the colimit of a $\lambda$-directed system $(f_i^\prime:A_i^\prime\to B_i^\prime \mid i\in I)$ of morphisms between $<\!\lambda$-presented modules. Let $\mathcal F$ be an arbitrary $\lambda$-complete filter on $I$ containing all the sets $\uparrow\! i = \{j\in I\mid i\leq j\}$ where $i$ runs through $I$. Then there exists a $\lambda$-pure embedding of $\Coker f$ into the reduced power $B^I/\mathcal F$.
\end{lemma}

\begin{proof} As in the proof of \cite[Proposition~2.30(ii)]{AR}, we use push-outs and $\lambda$-purity of $f$ to get a $\lambda$-directed system $\mathcal S = (f_i:A\to B_i, b_{ji}:B_i\to B_j\mid i\leq j\in I)$ in the coma-category $A\downarrow\ModR$ consisting of split monomorphisms and such that $\varinjlim \mathcal S = f$. Let us denote by $b_i:B_i\to B$, $i\in I$, the colimit maps (so $b_if_i = f$) and by $g_i:B_i\to A$ the retractions, i.e.\ $g_if_i = 1_A$.

Let $p:B\to B^I/\mathcal F$ be the diagonal embedding of $B$ into the reduced power. We also define the morphism $q:B\to B^I/\mathcal F$ by first picking, for each $b\in B$, an index $i\in I$ and $c\in B_i$ such that $b_i(c) = b$, and then setting $q(b) = [(u_j)_{j\in I}]_{\mathcal F}$ where $u_j = b_jf_jg_jb_{ji}(c)$ if $i\leq j$ and $u_j = 0$ otherwise. By the assumption on $\mathcal F$, this is a~correctly defined homomorphism.

It is easy to check that $pf = qf$, whence we get a unique morphism $\nu: \Coker f\to B^I/\mathcal F$ such that $p - q = \nu\pi$ where $\pi:B\to \Coker f$ is the canonical projection. It remains to verify that $\nu$ is a $\lambda$-pure monomorphism.

Let $\Omega$ be a system of $R$-linear equations with parameters from $\Coker f$ such that $\tau:=|\Omega|<\lambda$. Since $\mathcal S$ is $\lambda$-directed, there exists $i\in I$ such that each parameter in $\Omega$ is of the form $\pi b_i(c)$ for some $c\in B_i$. Let us enumerate in $(c_\alpha\mid \alpha<\tau)$ all these elements $c$. Assume that the system $\Omega$ with parameters $(\nu\pi b_i(c_\alpha)\mid \alpha<\tau)$ has a~solution in $B^I/\mathcal F$.
By $\lambda$-completeness of $\mathcal F$ and the other assumption we have put on it, there exists $j\in I, i\leq j$ such that $\Omega$ with parameters $(b_i(c_\alpha)- b_jf_jg_jb_{ji}(c_\alpha)\mid \alpha<\tau)$ has a~solution in $B$. Taking the $\pi$-image of this solution (note that $\pi b_jf_j = \pi f = 0$), we obtain a~desired solution of $\Omega$ in $\Coker f$, showing that $\nu$ is a $\lambda$-pure monomorphism.
\end{proof}

\begin{remark} $(1)$ Given a $\lambda$-pure monomorphism $f$, there is always a $\lambda$-directed system $(f_i^\prime\mid i\in I)$ consisting of morphisms between $<\lambda$-presented modules such that its directed colimit is $f$. Also note that the filter on $I$ generated by the sets $\uparrow i$ is $\lambda$-complete, whence there always exists a~filter $\mathcal F$ satisfying the assumption of Lemma~\ref{l:pureepi}.

\smallskip

\noindent $(2)$ Since Lemma~\ref{l:pureepi} holds also for $\lambda = \aleph_0$, it gives a direct proof of the well-known fact that each definable class of modules (i.e.\ closed under directed colimits, pure submodules and products) is closed under pure-epimorphic images. Note that the reduced power $B^I/\mathcal F$ is the directed colimit of powers of $B$ and canonical epimorphisms between them.
\end{remark}

Now we are ready to prove the promised result. Recall that, given a regular infinite cardinal $\kappa$, a category $\mathcal K$ is called \emph{$\kappa$-accessible}, if it has $\kappa$-directed colimits and there is a set $\mathcal S$ of $<\kappa$-presented objects such that every object is a $\kappa$-directed colimit of objects from $\mathcal S$. For instance, $\rmod{R}$ is an $\aleph_0$-accessible category for any ring $R$ since each module is a direct limit of finitely presented modules. 

\begin{theorem}\label{t:strongcomp1} Let $\kappa$ be a strongly compact cardinal and $R$ a ring such that $|R|<\kappa$. Then each $\kappa$-pure epimorphism $g:B\to C$ with $B$ projective splits. In particular, the category $\mathcal P_0$ of all projective $R$-modules and $R$-homomorphisms is $\kappa$-accessible.
\end{theorem}

\begin{proof} We can assume without loss of generality that $B$ is actually a free module. Using the lemma above for $\kappa = \lambda$, we can find a $\kappa$-pure embedding $h:C\to B^I/\mathcal U$ with $\mathcal U$ a~$\kappa$-complete ultrafilter; remember that each $\kappa$-complete filter extends to a~$\kappa$-complete ultrafilter. By \cite[Theorem~II.3.8]{EM}, the ultrapower $B^I/\mathcal U$ is a free module as well. We repeat the process with $\Coker h$, and so on, and eventually obtain an unbounded pure-exact complex consisting of projective (even free) modules which is well known to be contractible, cf.\ \cite[Theorem~2.5]{BG}. In particular, $C$ is projective and $g$ splits.

Finally, since each $\kappa$-directed colimit of projective modules is a $\kappa$-pure-epimorphic image of a direct sum of projective modules, we see that the class $\mathcal P_0$ is closed under taking $\kappa$-directed colimits. To see that $\mathcal P_0$ is $\kappa$-accessible, observe that each projective module is the directed union of a $\kappa$-directed system of $<\kappa$-generated projective modules.
\end{proof}

Another consequence of Lemma~\ref{l:pureepi} is the following closure property of the class $\mathcal D$ of all right flat Mittag-Leffler ($= \aleph_1$-projective) modules.

\begin{proposition}\label{p:FMLclosure} Let $R$ be a left coherent ring such that any (finite or infinite) intersection of finitely generated submodules of ${}_RR^{(\omega)}$ is finitely generated. Then $\mathcal D$ is closed under $\aleph_1$-pure-epimorphic images.
\end{proposition}

\begin{proof} Recall that $\mathcal D$ is always closed under pure submodules and directed colimits of $\aleph_1$-continuous directed systems, cf.\ \cite[Lemma~4.1, Proposition~2.2]{HT}. Furthermore, by \cite[Theorem~4.7]{HT}, $\mathcal D$ is closed under products precisely over rings satisfying our assumption. Let $0\to A\to B\to C\to 0$ be an $\aleph_1$-pure short exact sequence with $B\in\mathcal D$. It follows from Lemma~\ref{l:pureepi} that $C$ is a pure submodule of a~reduced power $B^I/\mathcal F$ where $\mathcal F$ is an $\aleph_1$-complete filter. Consequently, $B^I/\mathcal F$ is the directed colimit of $\aleph_1$-continuous directed system consisting of direct powers of $B$ and epimorphisms. As such $B^I/\mathcal F\in\mathcal D$, and hence also $C\in\mathcal D$.
\end{proof}

Below, we show that the additional assumption on $R$ is necessary.

\begin{example} In this example, by `countable', we mean `of cardinality $<\aleph_1$'. Consider the boolean ring $R$ consisting of all the subsets of $\omega_1$ which are either countable, or have countable complement, with the usual operations of symmetric difference and intersection. The ideal $I$ of all the countable subsets of $\omega_1$ is easily seen to be maximal, and it is not a direct summand in $R_R$. It follows that the simple module $R/I$ is not (flat) Mittag-Leffler. On the other hand, the inclusion $I\hookrightarrow R_R$ is the $\aleph_1$-directed (but not $\aleph_1$-continuous!) colimit of split inclusions $eR\hookrightarrow R_R$ where $e$ runs through countable subsets of $\omega_1$, thus $I\hookrightarrow R_R$ is $\aleph_1$-pure and $R/I$ is an $\aleph_1$-pure-epimorphic image of a free module which is not (flat) Mittag-Leffler.
\end{example}

\end{document}